\documentclass[journal,twoside,web]{ieeecolor}

\usepackage{generic}
\usepackage{bm}
\let\labelindent\relax
\usepackage{enumitem}
\usepackage{cite}
\usepackage{amsmath}
\usepackage{amssymb,amsfonts}

\usepackage{amsthm}
\usepackage[ruled,linesnumbered]{algorithm2e}
\usepackage{textcomp}
\usepackage{graphicx,color}
\usepackage{mathrsfs}
\usepackage{url}
\usepackage{color}
\usepackage{dsfont}
\usepackage{bbm}
\usepackage{booktabs}
\usepackage{array}
\usepackage[table]{xcolor}
\usepackage{tikz}
\usetikzlibrary{decorations.pathmorphing}
\tikzset{snake it/.style={decorate, decoration=snake}}

\usepackage{tcolorbox}
\usepackage{mathtools}
\usepackage{multicol}
\usepackage[colorlinks=true, linkcolor=blue, citecolor=blue]{hyperref}
\usepackage{accents}
\usepackage{epsfig}
\usepackage{lcsys}

\newtheorem{theorem}{Theorem}

\newtheorem{definition}{Definition}
\newtheorem{corollary}[theorem]{Corollary}

\newtheorem{problem}{Problem}
\newtheorem{remark}{Remark}
\newtheorem{assumption}{Assumption}
\newtheorem{example}{Example}

\newcommand{\real}{\mathbb{R}}

\renewcommand{\theenumi}{(\roman{enumi})}

\newcommand*{\QEDA}{\hfill\ensuremath{\blacksquare}}

\newcommand\oprocendsymbol{\hbox{$\square$}}
\newcommand\oprocend{\relax\ifmmode\else\unskip\hfill\fi\oprocendsymbol}

\title{\LARGE \textbf{Constricting Tubes for Prescribed-Time Safe Control}}

\author{Darshan Gadginmath \qquad Ahmed Allibhoy \qquad Fabio Pasqualetti 
\thanks{This work was supported in part by awards ARO-W911NF-24-1-0228 and NSF-CMMI-2308639. D.  Gadginmath is with Mitsubishi Electric Research Labs, Cambridge, MA, USA, 02139. A. Allibhoy and F. Pasqualetti are with the Department of Electrical Engineering and Computer Science at the University of California at Irvine, Irvine, CA, 92697, USA.
    E-mail: \href{mailto:gadginmath@merl.com}{\texttt{gadginmath@merl.com}},
    \href{mailto:aallibho@uci.edu}{\texttt{aallibho@uci.edu}},
    \href{mailto:fabiopas@uci.edu}{\texttt{fabiopas@uci.edu.}} }
}

\begin{document}

\maketitle
\thispagestyle{empty}

\begin{abstract}
We propose a constricting Control Barrier Function (CBF) framework
for prescribed-time control of control-affine systems with input
constraints. Given a system starting outside a target safe set, we
construct a time-varying safety tube that shrinks from a relaxed set
containing the initial condition to the target set at a user-specified
deadline. Any controller rendering this tube forward invariant
guarantees prescribed-time recovery by construction. The constriction
schedule is bounded and tunable by design, in contrast to
prescribed-time methods where control effort diverges near the
deadline. Feasibility under input constraints reduces to a single
verifiable condition on the constriction rate, yielding a closed-form
minimum recovery time as a function of control authority and initial
violation. The framework imposes a single affine constraint per
timestep regardless of state dimension, scaling to settings where
grid-based reachability methods are intractable. We validate on an 
18-dimensional multi-agent system, demonstrating scalability and prescribed-time recovery with bounded control effort.
\end{abstract}




\section{Introduction}

\IEEEPARstart{D}RIVING a dynamical system into a desired set within a prescribed time
is a recurring requirement where both timeliness and safety are
critical. Applications such as satellite rendezvous under a mission deadline, 
or steering a robot into a goal region before a planning horizon expires require explicit prescribed-time control. The
challenge is hardest when the system is nonlinear and subject to input
constraints, a combination not well addressed by existing tools.
Hamilton-Jacobi reachability~\cite{SB-MC-SH-CJT:2017,IMM-AMB-CJT:2005}
provides formal guarantees but requires solving a PDE over the state
space, scaling exponentially in dimension and precluding real-time
synthesis. Control Barrier Functions
(CBFs)~\cite{ADA-SC-etal:2019} enable real-time synthesis via
pointwise quadratic programs (QPs) and enforce forward invariance of a
safe set, but are designed for systems that already satisfy the safety
constraint. When the initial condition lies outside the target set,
standard CBF conditions only achieve recovery asymptotically.
Prescribed-time and fixed-time methods attempt to address this gap, but
share a common limitation. The convergence mechanism, whether a
singular time-varying gain~\cite{AP:2011}, a Lyapunov decay
condition~\cite{KG-EA-DP:2020_convergenceCLF}, or a slack
variable~\cite{KG-DP:2021}, is implicit in the analysis rather than an
explicit design parameter. As a result, it is difficult to bound
control effort a priori or verify feasibility under input constraints
at design time.

In this paper, we propose a \emph{constricting tube} framework that
reduces the prescribed-time recovery problem to a forward
invariance problem. Given an initial condition outside the target set,
we construct a time-varying tube that contains the initial state and
collapses exactly to the target set at the deadline. The tube is
parameterized by a designer-specified constriction schedule, bounded
and tunable by construction, so that any controller rendering the tube
forward invariant guarantees prescribed-time recovery. The constriction
rate directly characterizes the worst-case control demand and admits a
closed-form characterization of the minimum deadline without requiring
a Lyapunov analysis or offline computation. This yields a highly scalable framework for prescribed-time control synthesis with convergence certificates. 

\noindent\textbf{Related work.} The problem of stabilizing a target set has been studied from
various angles. We review the directions most relevant to constrained 
prescribed-time stabilization.

\emph{Hamilton-Jacobi (HJ) reachability.}
HJ reachability~\cite{SB-MC-SH-CJT:2017} provides formal safety and
reachability guarantees via Hamilton-Jacobi-Isaacs PDEs solved backward
in time. The backward reachable set (BRS)~\cite{JFF-MC-CJT-SSS:2015}
is the set of initial conditions from which $\mathcal{C}$ can be reached
at a prescribed terminal time $T$ under optimal control, making this
formulation most directly analogous to ours. However, solving the backward PDE requires discretizing the state 
space over a grid, and complexity scales as $\mathcal{O}(M^n)$ with state 
dimension $n$ and grid size $M$, limiting practical use to low-dimensional systems. The 
proposed framework, by contrast, imposes a single affine constraint 
per timestep and scales as $\mathcal{O}(n)$, making it applicable to 
high-dimensional settings. Neural approximations such
as DeepReach~\cite{SB-CJT:2021} extend HJ reachability to higher dimensions but
require offline training and provide only probabilistic guarantees.
The Control Barrier Value Function~\cite{JJC-etal:2021} recovers the
maximal controlled-invariant set while admitting a QP-based controller,
but does not provide an explicit mechanism for prescribed-time recovery
under input constraints. In contrast, the proposed framework imposes an affine constraint per timestep regardless of state dimension,
admits an analytic feasibility certificate, and is demonstrated on an
$18$-dimensional system, and higher in~\cite{DG-AA-FP:2026}, where grid-based methods are computationally intractable.

\emph{Prescribed-time stabilization via scaling functions.}
A prominent line of work achieves prescribed-time stabilization through
a time-varying scaling function that diverges unboundedly near
the deadline~\cite{YS-YW-JH-MK:2017, WL-MK:2023}. Fixed-time CLF
methods~\cite{AP:2011} and their CBF
extensions~\cite{KG-DP:2021} enforce convergence through Polyakov-type
decay conditions. The CLF-based prescribed-time framework
of~\cite{KG-EA-DP:2020_convergenceCLF} characterizes a feasibility
region through global quadratic bounds on the CLF, but these constants
are conservative and may not exist globally. Across all of these
approaches, the control gain grows without bound as time reaches the
deadline, and input constraints are handled through slack variables that
may sacrifice the timing guarantee~\cite[Remark~5]{KG-DP:2021}.

\emph{Prescribed-time and fixed-time CBFs.}
Several recent works modify the CBF condition with time-varying gains
to enforce temporal guarantees.~\cite{KG-etal:2021_multirate} combines
fixed-time CBFs with MPC-based planning under input constraints, but
does not explicitly characterize the minimum feasible
deadline.~\cite{JSK-SK:2025} achieves fixed-time convergence for
higher-order systems using a polynomial-in-time barrier, but does not
account for input bounds and control effort grows near the
deadline.~\cite{TYH-etal:2024} combines backstepping with a
time-varying blow-up function for unknown dynamics, but relies on a
class-$\mathcal{K}$ gain that diverges with time. A common limitation
across these methods is that the convergence mechanism is implicit in
the analysis, making it difficult to bound control effort a priori or
verify feasibility under input constraints.

Our framework makes the prescribed-time convergence mechanism explicit
through a designer-specified constricting tube. Under input 
constraints, we derive feasibility constraints for convergence under worst-case control demand. The framework also 
generalizes to stochastic systems, such as generative sampling 
models~\cite{DG-AA-FP:2026}.

\noindent\textbf{Contributions.} The main contributions of this paper are:
\begin{enumerate}
    \item A constricting CBF framework that guarantees prescribed-time
    recovery via forward invariance of a designer-specified
    safety tube, for any controller satisfying an affine constraint 
    independent of state dimension.
    \item A feasibility analysis under input constraints, yielding a
    verifiable design-time condition on the constriction rate, and a lower bound on the minimum recovery time.
\end{enumerate}
We further verify our feasibility analysis on an adaptive cruise control example, and validate scalability of our framework on an 18-dimensional 
coupled multi-agent system with 15 pairwise collision-avoidance  constraints.

\section{Constricting CBF framework}\label{sec:methodology}
\subsection{Problem formulation}
We consider control-affine systems
\begin{equation}
    \dot{x} = f(x) + g(x)u,
    \label{eqn:system}
\end{equation}
with state $x \in \mathbb{R}^n$, control $u \in \mathcal{U}$,
where $\mathcal{U} \subset \real^m$ is a set of admissible inputs, and locally Lipschitz vector fields $f :
\mathbb{R}^n \to \mathbb{R}^n$ and $g : \mathbb{R}^n \to \mathbb{R}^{n
\times m}$. Let $h : \mathbb{R}^n \to \mathbb{R}$ be a continuously
differentiable function characterizing a target set $\mathcal{C} = \{x :
h(x) \geq 0\}$. Given a deadline $T > 0$ and an initial condition outside the target set $x(0) \notin \mathcal{C}$, we seek a feedback control policy that solves the
following problem.

\begin{problem}[\bf Prescribed-time control]
\label{prob:main}
Given system~\eqref{eqn:system} with $x(0) \notin \mathcal{C}$,
synthesize a feedback policy $u = \kappa(x, t)$ such that:
\begin{enumerate}[label=(\roman*)]
    \item \emph{Prescribed-time recovery:} $x(T) \in \mathcal{C}$.
    \item \emph{Forward invariance:} $x(t) \in \mathcal{C}$ for all $t \geq T$.
\end{enumerate}
\end{problem}

We make the following assumptions throughout the paper.
\begin{assumption}[\textbf{Control Barrier Function and Regularity}]
\label{assn:standing}
\quad
\begin{enumerate}[label=(\roman*)]
    \item The function $h$ is a CBF for $\mathcal{C}$: there exists an
    extended class-$\mathcal{K}$ function $\gamma$ such that for all $x \in \mathbb{R}^n$
    \begin{equation}
        \sup_{u \in \mathbb{R}^m} [L_f h(x) + L_g h(x)\,u] \geq -\gamma(h(x)),
        \label{eqn:cbf}
    \end{equation}
    \item $L_g h(x) \neq 0$ for all $x \in \partial \mathcal{C}$.
\end{enumerate}
\end{assumption}

Assumption~\ref{assn:standing}(i) guarantees that $\mathcal{C}$ is
controlled invariant, i.e., there always
exists a control input keeping the system inside $\mathcal{C}$ once
it has entered. Assumption~\ref{assn:standing}(ii) is a technical 
assumption required to ensure feasibility of the optimization-based 
feedback controller we introduce in the sequel. 

We solve Problem 1 by defining a 
time-dependent relaxation of the target set that constricts over time, so
that any trajectory remaining inside the shrinking tube is driven
into $\mathcal{C}$ by the deadline. With this formulation, 
the problem reduces to maintaining 
forward invariance of the shrinking tube, as we describe next.

\subsection{Constricting tubes for prescribed-time control}
\label{sec:framework}
We begin by defining a time-varying tube that uses a designer-specified relaxation of set~$\mathcal{C}$ with the desired convergence timeline. 
\begin{definition}[\textbf{Constricting tube and relaxation schedule}]
\label{def:constricting_tube}
Given a CBF $h$ for $\mathcal{C}$ and an initial condition $x(0)$,
define the \emph{initial relaxation}
\begin{equation}
    r_0 \coloneqq \max\bigl\{0,\, -h(x(0))\bigr\},
    \label{eqn:r0}
\end{equation}
and the \emph{tube function}
$\tilde{h} : \mathbb{R}^n \times [0,T] \to \mathbb{R}$ as
\begin{equation}
    \tilde{h}(x, t) = h(x) + r(x(0), t),
    \label{eqn:constricting_barrier}
\end{equation}
where the \emph{relaxation schedule}
$r : \mathbb{R}^n \times [0,T] \to \mathbb{R}_{\geq 0}$ is
$\mathcal{C}^1$ in~$t$ and satisfies:
\begin{enumerate}[label=(\roman*)]
    \item \emph{Initial containment:} $r(x(0), 0) = r_0$,
    \item \emph{Target set recovery:} $r(x(0), T) = 0$,
    \item \emph{Monotone constriction:} $\dot{r}(x(0), t) \leq 0$
          for all $t \in [0,T]$.
\end{enumerate}
The associated \emph{constricting tube} is the superlevel set
\begin{equation}
    \tilde{\mathcal{C}}(t) = \{x \in \mathbb{R}^n : \tilde{h}(x,t) \geq 0\},
    \label{eqn:tube}
\end{equation}
satisfying $x(0) \in \tilde{\mathcal{C}}(0)$ and
$\tilde{\mathcal{C}}(T) = \mathcal{C}$.
\end{definition}

The schedule $r(x(0), t)$ is parameterized by the initial condition,
and the entire schedule is determined once~$x(0)$ is fixed. The
definition unifies two cases. When $x(0) \notin \mathcal{C}$, we have
$r_0 = -h(x(0)) > 0$, and the relaxation inflates the target set so that
$x(0) \in \tilde{\mathcal{C}}(0)$. When $x(0) \in \mathcal{C}$, we have
$r_0 = 0$, so $r(x(0),t) = 0$ for all $t$ and the tube reduces to
$\tilde{\mathcal{C}}(t) = \mathcal{C}$, recovering standard CBF-based
forward invariance as a special case. The relaxation schedule, $r(x(0),t)$, is a design choice independent
of the system dynamics. Possible choices include:

\begin{itemize}
    \item Linear schedule: $r(x(0),t) = r_0(1-t/T)$
    \item Polynomial schedule: $r(x(0),t) = r_0(1-t/T)^p$, $p > 0$
    \item Exponential schedule: $r(x(0),t) = r_0\bigl(e^{\lambda(1-t/T)} - 1\bigr)/(e^\lambda - 1)$, $\lambda > 0$
\end{itemize}


\begin{figure}
    \centering
    \includegraphics[width=0.8\linewidth]{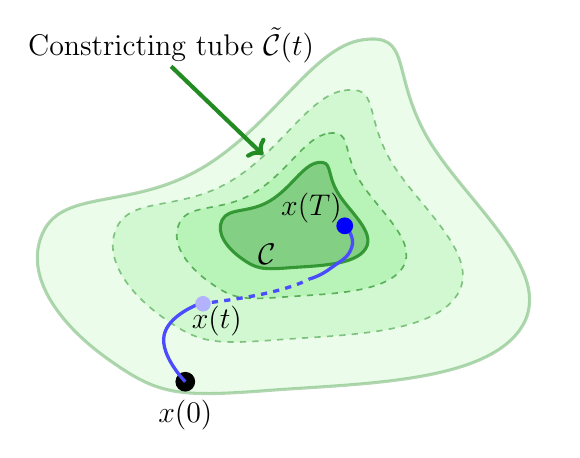}
    \vspace{-0.0ex}
    \caption{Geometry of the constricting tube framework. The initial
    condition $x(0)$ lies outside the target set
    $\mathcal{C} = \{x : h(x) \geq 0\}$. The schedule $r(x(0),t)$
    inflates $\mathcal{C}$ into a tube $\tilde{\mathcal{C}}(t) = \{x :
    h(x) + r(x(0),t) \geq 0\}$ containing $x(0)$ at $t=0$. As $r(x(0),t)
    \to 0$, the tube boundary constricts through the level sets of $h$
    until $\tilde{\mathcal{C}}(T) = \mathcal{C}$. Any controller
    rendering $\tilde{\mathcal{C}}(t)$ forward invariant guarantees
    $x(T) \in \mathcal{C}$.}
    \label{fig:constricting_cbf}
    \vspace{-0.0ex}
\end{figure}

Solving Problem~\ref{prob:main} requires enforcing forward invariance
of~$\tilde{\mathcal{C}}(t)$. If the trajectory never leaves
$\tilde{\mathcal{C}}(t)$, then~$x(T) \in \tilde{\mathcal{C}}(T) = \mathcal{C}$,
achieving the control objective. The following definition, paralleling
the standard CBF condition~\eqref{eqn:cbf}, identifies when such forward
invariance is achievable.

\begin{definition}[\textbf{Constricting CBF}]
\label{def:constricting_cbf}
The tube function $\tilde{h}$ in~\eqref{eqn:constricting_barrier} is a
\emph{constricting CBF} for $\mathcal{C}$ on $[0,T]$, 
with relaxation schedule $r$ and initial condition $x(0)$,  if there exists an
extended class-$\mathcal{K}$ function $\gamma$ such that for all
$(x,t) \in \tilde{\mathcal{C}}(t) \times [0,T]$,
\begin{equation}
    \sup_{u \in \mathcal{U}}\,\bigl\{\, L_f h(x) + L_g h(x)\, u + \dot{r}(x(0),t)\,\bigr\}
    \geq -\gamma\bigl(\tilde{h}(x,t)\bigr).
    \label{eqn:cbf_condition}
\end{equation}
\end{definition}

Definition~\ref{def:constricting_cbf} extends the standard CBF
condition~\eqref{eqn:cbf} by the inward-pressure $\dot{r}(x(0),t) \leq 0$, of the constriction tube. When $r \equiv 0$ (i.e., $x(0) \in \mathcal{C}$), we recover the
standard CBF condition.



\begin{theorem}[\textbf{Prescribed-time recovery via tube invariance}]
\label{thm:forward_invariance}
Let $\tilde{h}$ be a constricting CBF for $\mathcal{C}$ on $[0,T]$ in the
sense of Definition~\ref{def:constricting_cbf}. Then any locally Lipschitz
feedback controller $u = \kappa(x, t)$ satisfying
\begin{equation}
    L_f h(x) + L_g h(x)\,\kappa(x,t) + \dot{r}(x(0),t)
    \geq -\gamma\bigl(\tilde{h}(x,t)\bigr)
    \label{eqn:cbf_pointwise}
\end{equation}
for all $(x,t) \in \tilde{\mathcal{C}}(t) \times [0,T]$ renders
$\tilde{\mathcal{C}}(t)$ forward invariant on $[0,T]$ and ensures:
\begin{enumerate}[label=(\roman*)]
    \item \emph{Prescribed-time recovery:} $x(T) \in \mathcal{C}$.
    \item \emph{Forward invariance:} $x(t) \in \mathcal{C}$ for all $t \geq T$.
\end{enumerate}
\end{theorem}

\noindent\textbf{Proof.}
By Definition~\ref{def:constricting_tube}(i), $r(x(0),0) = r_0
\geq -h(x(0))$, so $\tilde{h}(x(0),0) \geq 0$ and
$x(0) \in \tilde{\mathcal{C}}(0)$. Augment the state as $z = (x,t)$
with dynamics $\dot{z} = [f(x)+g(x)\kappa(x,t),\,1]^\top$ and define the
space-time set $\tilde{\mathcal{S}} = \{(x,t) : \tilde{h}(x,t) \geq 0\}$.
On $\partial\tilde{\mathcal{S}}$, where $\tilde{h}(x,t) = 0$, the
gradient is $\nabla_{(x,t)} \tilde{h} = (\nabla h(x),\,\dot{r}(x(0),t))$. This is nonzero on $\partial\tilde{\mathcal{S}}$ since either
$\dot{r}(x(0),t) < 0$ under strict constriction, or
$\nabla h(x) \neq 0$ on $\partial\mathcal{C}$ (Assumption~\ref{assn:standing}(ii)) holds. Hence $0$ is a regular value of $\tilde{h}$ on
$\partial\tilde{\mathcal{S}}$. Condition~\eqref{eqn:cbf_pointwise} then
reduces on $\partial\tilde{\mathcal{S}}$ to
$L_f h + L_g h\,\kappa + \dot{r}(x(0),t) \geq 0$,
which is the Nagumo tangency condition for $\tilde{\mathcal{S}}$.
By~\cite[Section~4.2.2]{FB-SM:08}, $x(0) \in \tilde{\mathcal{C}}(0)$
implies $\tilde{h}(x(t),t) \geq 0$ for all $t \in [0,T]$, i.e.,
$\tilde{\mathcal{C}}(t)$ is forward invariant on $[0,T]$.
At $t = T$, Definition~\ref{def:constricting_tube}(ii) gives
$r(x(0),T) = 0$, so $h(x(T)) = \tilde{h}(x(T),T) \geq 0$,
proving~(i). For $t \geq T$, condition~\eqref{eqn:cbf_pointwise} with
$r \equiv 0$ reduces to the standard CBF condition~\eqref{eqn:cbf}, proving (ii). \hfill$\blacksquare$

The constriction rate $\dot{r}(x(0),t) < 0$ acts as an inward pressure
beyond what is needed for forward invariance of a static set. This
additional demand drives convergence. The faster the tube constricts,
the harder the controller must push the state inward. Since~$r$ is
$\mathcal{C}^1$ on the compact interval~$[0,T]$, $\dot{r}(x(0),t)$ is
bounded, and the additional demand is finite. A feedback controller
can be synthesized by standard optimization techniques like
in~\cite{ADA-SC-etal:2019}.

\section{Feasibility under input constraints}
\label{sec:feasibility}
We synthesize here a
feedback controller for input constrained systems possessing 
a candidate constricting CBF~$\tilde{h}$. 
We assume
the input constraint is $\|u\| \leq u_{\max}$, but our
analysis generalizes to $u \in \mathcal{U}$ where $\mathcal{U}\subseteq \mathbb{R}^m$ is convex. The proposed controller is 
\begin{align}
\begin{aligned}
    &\kappa(x,t) = \underset{{u \in \mathbb{R}^m}}{\text{argmin}} \quad  J(x,u,t) \\
    \text{s.t.} \quad
    & L_f h(x) + L_g h(x)\,u + \dot{r}(x(0),t) \geq
      -\gamma\bigl(\tilde{h}(x,t)\bigr), \\
    & \|u\| \leq u_{\max}.
\end{aligned}
\label{eqn:qp}
\end{align}
When $J$ is convex and quadratic in $u$,~\eqref{eqn:qp} is a convex
program solvable in real time via standard solvers. 
The program is feasible 
if $\tilde{h}$ is a constricting CBF, and 
we derive here sufficient conditions for feasibility to hold.
For $t \geq T$, the constriction schedule 
satisfies $\dot{r}(x(0), t) = 0$,
and the program reverts to the standard optimization-based safe controller with CBF constraints.

Theorem~\ref{thm:forward_invariance} guarantees prescribed-time recovery
whenever the problem~\eqref{eqn:qp} is feasible.
In the absence of input constraints, 
\eqref{eqn:cbf_condition}
is feasible for all $(x, t) \in \tilde{\mathcal{C}} \times [0, T]$ by Assumption~\ref{assn:standing}(ii). 
However, the presence of input constraints fundamentally changes this
picture. When the control authority available to
satisfy~\eqref{eqn:cbf_condition} is limited, the constriction
demand $|\dot{r}(x(0),t)|$ may exceed what the actuators can deliver
at certain states or times. In this setting, there is a tradeoff 
between the amount of control authority~$u_{\text{max}}$ and the deadline 
$T$: as $u_{\text{max}}$ decreases, one must tolerate a longer deadline in order 
to maintain feasibility. 

To choose a feasible deadline, it is crucial to verify whether $u_{\max}$
is large enough to meet the constriction demand $|\dot{r}(x(0),t)|$ at
every point on the tube. We define the \emph{barrier authority} to capture this idea.
\begin{equation}
    \sigma(x) \coloneqq \sup_{u \in \mathcal{U}} \dot{h}(x,u).
    \label{eq:safety_margin}
\end{equation}
For $\mathcal{U} = \{u : \|u\| \leq u_{\max}\}$, this evaluates to
$\sigma(x) = \|L_g h(x)\|\, u_{\max} + L_f h(x)$. The barrier authority
$\sigma(x)$ is the maximum rate that~$h$ can be increased at $x \in \mathbb{R}^n$.
When $\mathcal{U} = \mathbb{R}^m$, $\sigma(x) = +\infty$
wherever $L_g h(x) \neq 0$, 
indicating there is no limit on control authority in the absence 
of input constraints.


\begin{example}[\textbf{Barrier authority for inverted pendulum}]To illustrate the geometry of the barrier authority, consider the
controlled simple pendulum 
\begin{align*}
    \dot{x}_1 = x_2, \qquad
    \dot{x}_2 = -\sin x_1 + u,
\end{align*}
with CBF $h(x) = c - (x_1 - \pi)^2 - x_2^2$, a ball of radius $\sqrt{c}$ around
the unstable upright equilibrium $(\pi, 0)$.
The barrier authority is
\begin{equation}
  \sigma(x) = 2|x_2|\,u_{\max} + 2x_2\bigl(\sin x_1 - (x_1 - \pi)\bigr).
\end{equation}
Figure~\ref{fig:barrier_authority} shows $\sigma(x)$ over the domain
$[\pi - 2.5,\, \pi + 2.5] \times [-2.5,\, 2.5]$ for $c = 0.01$ and
$u_{\max} = 1.5$. The sign of $\sigma$ depends only on $x_1$: for each half-plane, $\sigma$ factors as
$2x_2 \cdot \varphi(x_1)$ where $\varphi(x_1) = u_{\max} + \sin x_1 - (x_1 - \pi)$
for $x_2 > 0$, and $-\varphi(x_1)$ for $x_2 < 0$, so the $\sigma = 0$ boundary
is the vertical line $x_1 = \pi \pm 0.79$ rad, independent of $|x_2|$. 
The region with a negative barrier authority (red) occupies two lobes. In both lobes, the drift $L_f h = 2x_2(\sin x_1 - (x_1 - \pi))$ is large and
negative: the pendulum is both displaced from upright (so gravity destabilizes) and moving away from it (so velocity compounds the drift),
and together they exceed the available control authority $u_{\max}$.
  
\begin{figure}[t]
  \centering
  \includegraphics[width=0.9\linewidth]{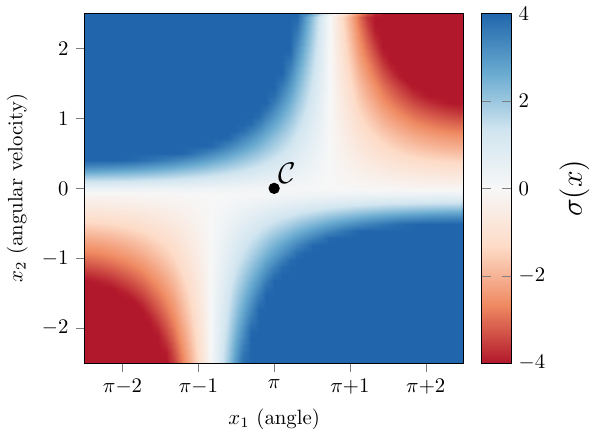}
  \caption{Barrier authority $\sigma(x)$ for the controlled pendulum
    targeting the upright equilibrium $(\pi,0)$ (dot) with $u_{\max}=1.5$.
    Blue: $\sigma > 0$; red: $\sigma < 0$.
    The red lobes correspond to states where the pendulum is
    displaced from upright and moving further away, so gravitational
    drift and velocity compound to reduce $\sigma_{\min}$.}
  \label{fig:barrier_authority}
\end{figure}

 \end{example}

If $\sigma(x) \leq 0$, even maximum control effort cannot increase $h$. At such points, the CBF condition will be infeasible. 
Let $\sigma_{\min}$ be the \emph{worst-case barrier authority},
\begin{equation}
    \sigma_{\min} \coloneqq
    \inf_{\substack{t \in [0,T] \\ x \in \tilde{\mathcal{C}}(t)}}
    \sigma(x).
    \label{eq:sigma_min}
\end{equation}
The worst-case barrier authority captures the most
demanding point on the tube, which is the state and time at which the system
has the least capacity to push $h$ upward, while the constriction
schedule is still active. Feasibility of~\eqref{eqn:qp} at $(x,t) \in \tilde{\mathcal{C}}(t) 
\times [0,T]$ reduces to comparing the constriction demand 
$|\dot{r}(x(0),t)|$ against the local barrier authority $\sigma(x)$. 
When the constriction rate is too aggressive relative to the available 
control authority, no admissible $u$ can satisfy the CBF condition, 
and the program becomes infeasible. When $\sigma_{\min} \leq 0$, a longer deadline or a less aggressive  constriction schedule may be required to restore feasibility. The following result characterizes this 
tradeoff precisely, linking the deadline $T$, the constriction schedule 
$r(x(0),t)$, and the bound $u_{\max}$ in a single verifiable condition.

\begin{theorem}[\textbf{Feasibility and minimum recovery time}]
\label{thm:feasibility}
Consider system~\eqref{eqn:system} with the relaxation 
schedule $r(x(0), t)$ and constricting CBF \eqref{eqn:constricting_barrier}. 
\begin{enumerate}[label=(\roman*)]
    \item \emph{Local feasibility:} The
    program~\eqref{eqn:qp} is feasible at $(x,t) \in
    \tilde{\mathcal{C}}(t) \times [0,T]$ if and only if
    \begin{equation}
        |\dot{r}(x(0),t)| \leq \sigma(x).
        \label{eqn:feasibility}
    \end{equation}
    \item \emph{Global feasibility (linear schedule):} For the linear
    schedule $r(x(0),t) = r_0(1 - t/T)$, program~\eqref{eqn:qp} is feasible for
    all $(x,t) \in \tilde{\mathcal{C}}(t) \times [0,T]$ if and
    only if
    \begin{equation}
        T \geq T_{\min} \coloneqq \frac{r_0}{\sigma_{\min}},
        \label{eqn:tmin}
    \end{equation}
    where $r_0$ is the initial violation~\eqref{eqn:r0} and
    $\sigma_{\min}$ is the worst-case barrier
    authority~\eqref{eq:sigma_min}.
\end{enumerate}
\end{theorem}
\noindent\textbf{Proof.}
On the boundary $\partial\tilde{\mathcal{C}}(t)$, $\tilde{h} = 0$ so
$\gamma(\tilde{h}) = 0$ and the constraint~\eqref{eqn:cbf_condition}
is most restrictive. In the interior, $\tilde{h} > 0$ so
$-\gamma(\tilde{h}) < 0$, and any $u$ satisfying the boundary
condition also satisfies the interior condition. It is therefore
sufficient to analyze feasibility on $\partial\tilde{\mathcal{C}}(t)$.
\emph{(i) Local feasibility.} On $\partial\tilde{\mathcal{C}}(t)$,
$\gamma(\tilde{h}) = 0$ and~\eqref{eqn:cbf_condition} reduces to
\begin{equation}
    L_g h(x)\,u \geq -L_f h(x) + |\dot{r}(x(0),t)|,
    \label{eq:boundary_condition}
\end{equation}
where we used $\dot{r}(x(0),t) \leq 0$. The maximum of $L_g h(x)\,u$
over $\|u\| \leq u_{\max}$ is $\|L_g h(x)\| u_{\max}$, achieved by
aligning $u$ with $L_g h(x)^\top$. A feasible $u$ exists if and only
if $\|L_g h(x)\| u_{\max} \geq -L_f h(x) + |\dot{r}(x(0),t)|$, which
rearranges to~\eqref{eqn:feasibility}. Conversely, if
$|\dot{r}(x(0),t)| > \sigma(x)$, then no $u$ with $\|u\| \leq u_{\max}$
can satisfy~\eqref{eq:boundary_condition}, so~\eqref{eqn:qp} is infeasible.

\emph{(ii) Global feasibility.} For the linear schedule,
$|\dot{r}(x(0),t)| = r_0/T$ is constant. Global feasibility requires
the local condition~\eqref{eqn:feasibility} to hold for every $t \in
[0,T]$ and all $x \in \partial\tilde{\mathcal{C}}(t)$, i.e.,
$r_0/T \leq \sigma_{\min}$. Rearranging gives~\eqref{eqn:tmin}.
Conversely, if $T < T_{\min}$ then $r_0/T > \sigma_{\min}$, so the
local condition fails at the worst-case point on the tube boundary,
and~\eqref{eqn:qp} is infeasible there.\hfill$\blacksquare$

The bound $T_{\min} = r_0/\sigma_{\min}$ depends on the initial
condition through $r_0$ and, implicitly, through $\sigma_{\min}$ since
the tube boundary $\partial\tilde{\mathcal{C}}(t)$ is initialized from
$x(0)$.  Because the tube only constricts, this worst case reduces to a static infimum, $\sigma_{\min} = \inf_{x \in \tilde{\mathcal{C}}(0)} \sigma(x)$, over the initial inflated set.  A larger initial violation $r_0$ or lower control authority
$\sigma_{\min}$ each increase $T_{\min}$, capturing the intuitive
tradeoff. When $u_{\max} = \infty$, $\sigma_{\min} = \infty$ and
$T_{\min} = 0$, consistent with the unconstrained case where any
constriction rate is satisfiable. A similar analysis can be carried out for any convex input constraint set $\mathcal{U}$ by replacing $\|L_gh(x)\|u_{\max}$ in~\eqref{eq:safety_margin} 
with $\max_{u \in \mathcal{U}} L_gh(x)u$.
Crucially, $T_{\min}$ is computable before deploying the controller,
providing a design-time certificate for prescribed-time safety. 

The constriction schedule $r(x(0),t)$ gives the designer a further
handle beyond $T$. For a fixed $T > T_{\min}$, different schedules
distribute the demand $|\dot{r}(x(0),t)|$ differently over $[0,T]$.
The linear schedule imposes a constant pressure $r_0/T$ throughout;
an exponential schedule concentrates the effort early, reducing peak
demand near $t = T$. A polynomial schedule defers effort
toward the deadline. The local feasibility
condition~\eqref{eqn:feasibility} makes this tradeoff precise. The
schedule is feasible at time $t$ if and only if
$|\dot{r}(x(0),t)| \leq \sigma(x)$ pointwise, so the designer can
shape $r(x(0),t)$ to match the time-varying control authority of
the system. 

Unlike methods that introduce slack variables and verify feasibility
post-hoc~\cite{KG-DP:2021}, condition~\eqref{eqn:feasibility} is
verifiable at design time directly from $f$, $g$, $h$, and $u_{\max}$. For general nonlinear systems, computing $\sigma_{\min}$ a priori amounts to globally minimizing $\sigma$
 over $\tilde{\mathcal{C}}(0)$, which is tractable for structured cases, e.g.\ via sum-of-squares bounds for polynomial data, or in closed form when $\sigma$ is monotone over the reachable state range (Section~\ref{sec:acc}), and otherwise verified online. For linear systems with a quadratic
barrier, $T_{\min}$ admits a closed-form expression.

\begin{corollary}[\textbf{Closed-form $T_{\min}$ for linear systems}]
\label{cor:linear_tmin}
Consider the system $\dot{x} = Ax + Bu$, and $h(x) = c - x^\top P x$ with
$c > 0$ and $P = P^\top \succ 0$. The tube boundary
$\partial\tilde{\mathcal{C}}(t) = \{x : x^\top P x = c + r(x(0),t)\}$
is an ellipsoid of radius $\rho(t) = \sqrt{c + r(x(0),t)}$. On this boundary,
the barrier authority~\eqref{eq:safety_margin} evaluates to
\begin{equation}
    \sigma(x) = 2\bigl(\|B^\top P x\|\,u_{\max} - x^\top P A x\bigr),
    \label{eqn:sigma_linear}
\end{equation}
and its infimum over the tube horizon is achieved at $t = T$:
\begin{equation}
    \sigma_{\min} = 2c\left(\mu_{\min}\, u_{\max} - \lambda_{\max}(P^{1/2}AP^{-1/2})\right),
    \label{eqn:sigma_min_linear}
\end{equation}
where $\mu_{\min} := \min_{\|P^{1/2}v\|=1} \|B^\top P^{1/2} v\|$
is the minimum input gain over the ellipsoid. The minimum recovery time is then $T_{\min} = r_0/\sigma_{\min}$.
\end{corollary}
We provide an overview of the proof. The barrier authority for the linear system is~\eqref{eqn:sigma_linear}. On the tube boundary $x^\top Px = \rho(t)^2$, substitute $x = \rho(t)P^{-1/2}v$ with $\|v\| = 1$. After substitution, $\sigma(x) = 2\rho(t)\bigl(\|B^\top P^{1/2}v\|
\,u_{\max} - v^\top P^{1/2}AP^{-1/2}v\bigr)$, which scales linearly
in $\rho(t)$. Since $\rho(t)$ is monotone decreasing, the infimum
over the full tube is attained at $t = T$ where $\rho(T) = \sqrt{c}$,
and minimizing over $\|v\|=1$ yields~\eqref{eqn:sigma_min_linear}. The expression~\eqref{eqn:sigma_min_linear} is positive if and only
if $\mu_{\min} u_{\max} > \lambda_{\max}(P^{1/2}AP^{-1/2})$, i.e.,
the control authority exceeds the worst-case drift on the ellipsoid.
If this condition fails, recovery in finite time is not possible.

\begin{remark}[\textbf{Extension to higher relative degree}]
\label{rem:relative_degree}
When $h$ has relative degree $k > 1$ with respect
to~\eqref{eqn:system}, the framework extends by combining the
constricting barrier with the high-order CBF construction
of~\cite{WX-CB:2021}: define $\psi_0(x,t) = h(x) + r(x(0),t)$ and
propagate through the cascade $\psi_i = \dot{\psi}_{i-1} +
\gamma_i(\psi_{i-1})$ until the control input appears at order $k$. The
schedule $r(x(0),\cdot)$ must be $\mathcal{C}^k$, and the constriction
pressure enters through the $k$-th time derivative $r^{(k)}(x(0),t)$ in the
final constraint. An example for a system with relative degree 2 is provided in Section~\ref{sec:exp2}. \oprocend
\end{remark}

\begin{remark}[\textbf{Relation to HJ reachability}]~\label{rem:HJreach}
Hamilton-Jacobi reachability~\cite{SB-MC-SH-CJT:2017} provides globally optimal state-dependent certificates by solving a backward PDE: the value function $V(x, t)$ encodes the optimal worst-case demand at every $(x, t)$. The constricting schedule $\tilde h(x(0), t)$ is, by contrast, a closed-form surrogate 
that is more conservative
because the schedule is fixed based on the 
initial condition $x(0)$ and deadline $T$, and
does not adapt to the realized trajectory. 
HJ reachability and the constricting tube framework occupy complementary regimes: HJ provides optimal certificates 
whereas the constricting tube provides explicit feasibility certificates and QP-based synthesis in regimes where HJ is intractable, such as high-dimensional systems. 
Quantifying the gap between the constricting tubes
and the optimal rate of constriction, 
is a direction for future work.  \oprocend
\end{remark}

\begin{remark}[\textbf{Recursive feasibility}]\label{rem:recursive}
Theorem~\ref{thm:feasibility} establishes 
feasibility of \eqref{eqn:qp} at $(x, t) \in \tilde{\mathcal{C}} \times [0, T]$, but 
not recursive feasibility: the controller 
may steer the state
to a point where \eqref{eqn:qp} is infeasible, 
and no admissible control exists at future times. Addressing recursive feasibility requires global
reasoning over the full horizon, which can only be done in specific cases. This remains a direction for future work.
\oprocend
\end{remark}

\section{Numerical experiments}
\label{sec:simulations}
We demonstrate our framework on a testbed of experiments to validate its scalability and practical use\footnote{Code repository:
\url{https://github.com/darshangm/prescribed_time_control}.}.
We use $\gamma(\tilde{h}) = \alpha\tilde{h}$, with $\alpha = 0.9$ throughout.

\subsection{Prescribed-time recovery for adaptive cruise control}
\label{sec:acc}

We illustrate the feasibility certificate of Theorem~\ref{thm:feasibility} on
the adaptive cruise control (ACC) system of~\cite{ADA-XX-JWG-PD:2016}. We adopt the model and
parameters of~\cite{ADA-XX-JWG-PD:2016}, with a constant lead vehicle speed. We pose a prescribed-time \emph{recovery} task, in which the follower
begins at an unsafe headway and must restore the safe headway by a
deadline. The state is $x = (v_f, D)$, where $v_f$ is the follower speed and $D$ the
headway to a lead vehicle traveling at constant speed $v_0$. The dynamics are,
\begin{equation}
  \dot v_f = -\tfrac{1}{M} F_r(v_f) + \tfrac{1}{M} u,
  \qquad
  \dot D = v_0 - v_f,
  \label{eq:acc-dyn}
\end{equation}
with wheel force $u$ subject to $|u| \le u_{\max}$ and aerodynamic drag
$F_r(v_f) = f_0 + f_1 v_f + f_2 v_f^2$. We use the parameters
of~\cite{ADA-XX-JWG-PD:2016}: $M = 1650$\,kg, $f_0 = 0.1$\,N, $f_1 = 5$\,N\,s/m,
$f_2 = 0.25$\,N\,s$^2$/m$^2$, $g = 9.81$\,m/s$^2$, and the comfort-bounded
force $u_{\max} = a_f M g$ with $a_f = 0.25$, giving $u_{\max} \approx
4047$\,N. The time-headway barrier is
\begin{equation}
  h(x) = D - \tau_d v_f, \qquad \tau_d = 1.8\ \text{s},
  \label{eq:acc-barrier}
\end{equation}
 Since $L_g h = -\tau_d/M$ is constant and nonzero, Assumption~\ref{assn:standing}(ii)
holds globally and $h$ has relative degree one. 
The barrier authority~\eqref{eq:safety_margin} evaluates to
\begin{equation}
  \sigma(x) = \frac{\tau_d}{M}\, u_{\max} + (v_0 - v_f)
            + \frac{\tau_d}{M}\, F_r(v_f),
  \label{eq:acc-sigma}
\end{equation}
which depends on $v_f$ alone and is quadratic in $v_f$ with vertex far outside
any operating speed range. Hence $\sigma$ is monotone over the reachable speed
band, and the worst-case authority $\sigma_{\min}$~\eqref{eq:sigma_min} is
attained in closed form at the band endpoint. The follower starts at
$x(0) = (24, 30)$ behind a lead vehicle at $v_0 = 20$\,m/s, an unsafe
configuration with $h(x(0)) = -13.2 < 0$, so $r_0 = 13.2$. Over the reachable
band $v_f \in [v_0, v_f(0)]$ we obtain $\sigma_{\min} = 0.70$, hence $T_{\min}$ is
\begin{equation}
  T_{\min} = \frac{r_0}{\sigma_{\min}} = 18.79\ \text{s},
  \label{eq:acc-Tmin}
\end{equation}
computed entirely at design time, without simulating the closed loop. We apply the constricting QP~\eqref{eqn:qp} with the linear schedule
$r(t) = r_0(1 - t/T)$, $\alpha = 0.9$, and deadline
$T = 1.05\,T_{\min} = 19.73$\,s. Figure~\ref{fig:acc-sweep}\emph{(left)} shows the barrier
$h(x(t))$ rising from $-13.2$ and tracking the tube floor $-r(t)$, reaching
$\mathcal{C}$ at the deadline with $h(x(T)) \ge 0$. The follower decelerates to open the headway and the input remains within $u_{\max}$ throughout. To verify that the closed-form bound~\eqref{eq:acc-Tmin} is tight, we sweep the deadline $T$ across $T_{\min}$ and record the
feasibility margin $\min_t\,(\sigma(x(t)) - |\dot r(t)|)$.
Figure~\ref{fig:acc-sweep}\emph{(Right)} shows that this margin crosses zero exactly at
$T = T_{\min}$: for $T < T_{\min}$ the condition~\eqref{eqn:feasibility}
is violated and the QP becomes infeasible. The certificate $T \ge T_{\min}$ therefore predicts the exact feasibility boundary from the system parameters. 

\begin{figure}[t]
  \centering
  \begin{tikzpicture}

  \node at (-2.5,0) {
  \includegraphics[width=0.38\linewidth]{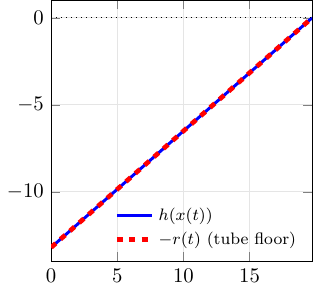}};
  \node at (-2.2,-1.8) {Time $t$};
  
  \node at (2.2,0) {\includegraphics[width=0.40\linewidth]{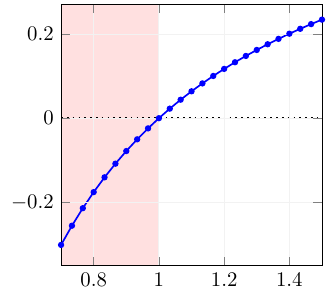}};
  \node at (2.3,-1.8) {$T / T_{\min}$};
  \node at (0.3,0) [rotate=90] {$\displaystyle\min_{t}\,\bigl(\sigma(x(t)) - |\dot r(t)|\bigr)$};
  
  \end{tikzpicture}
  \vspace*{-0ex}
  \caption{Prescribed-time recovery for adaptive cruise control.\emph{(Left)} The follower
    starts at an unsafe headway ($h(x(0)) = -13.2$) behind a slower lead and
    recovers the safe time-headway set $\mathcal{C}$ by the deadline
    $T = 19.73$\,s. The barrier $h(x(t))$ tracks the tube floor $-r(t)$ and
    enters $\mathcal{C}$ at $T$. \emph{(Right)} Feasibility margin $\min_t\,(\sigma(x(t)) - |\dot r(t)|)$ as the
    deadline $T$ is swept across $T_{\min}$. The margin crosses zero exactly at
    $T = T_{\min} = 18.79$\,s: the local feasibility
    condition~\eqref{eqn:feasibility} fails for $T < T_{\min}$ (shaded,
    infeasible).}
    \vspace*{-0ex}
  \label{fig:acc-sweep}
\end{figure}

\subsection{Coupled multi-agent prescribed-time reach}
\label{sec:hex}
We demonstrate the scalability of our framework on a coupled multi-agent system with
pairwise safety constraints. Consider $N = 6$ unicycle agents with the dynamics
\begin{equation}
\dot p_{x,i} = v_i \cos(\theta_i) \quad \dot p_{y,i} = v_i \sin(\theta_i) \quad \dot\theta_i = \omega_i
\end{equation}
 and input constraints $|v_i| \leq 2.5$~m/s and $|\omega_i| \leq 4.0$~rad/s. The state and input are $X\in \mathbb{R}^{18}, U\in \mathbb{R}^{12}$. Each agent is assigned a vertex of a
regular hexagon. Initial positions are sampled randomly inside the target hexagon. The barrier is
\begin{equation}
h(X) = \sum_{i=1}^{N} \bigl[ \epsilon^2 - \|p_i - p^\star_i\|^2
- \lambda(1 - \cos(\theta_i - \theta^\star_i)) \bigr],
\label{eq:joint-reach}
\end{equation}
with $\epsilon = 0.4$~m and $\lambda = 0.2$. The orientation target keeps
the relative degree at one. Collision-avoidance barriers are
$h_{ij}(X) = \|p_i - p_j\|^2 - d_{\min}^2$ with $d_{\min} = 0.55$~m.
 
We use an exponential constriction $r(x(0), t) = r_0 (e^{\lambda_s(1 - t/T)} - 1)/(e^{\lambda_s} - 1)$
with $\lambda_s = 3$ and $T = 17$~s, and $T=20$~s. We solve the following QP at every time instance,
\begin{equation}
\begin{aligned}
\min_{U \in \mathbb{R}^{12}} \quad & \|U - U_{\mathrm{nom}}(X)\|^2 \\
\text{s.t.} \quad & L_f h + L_g h\, U + \dot r \geq -\alpha(h + r), \\
& L_g h_{ij}\, U \geq -\alpha_c\, h_{ij}, \quad \forall (i,j), \\
& |v_i| \leq v_{\max}, \ |\omega_i| \leq \omega_{\max}, \ \forall i,
\end{aligned}
\label{eq:central-qp}
\end{equation}
with $\alpha = 0.9$, $\alpha_c = 1.9$, and $\beta = 1.0$. Here
$U_{\mathrm{nom}}$ is a slow PD controller,
so the constricting CBF acts as a safety filter for
prescribed-time recovery. 
 
Figure~\ref{fig:hex-overlay} illustrates the  trajectories. The constricting CBF controller drives all six agents
into their target balls at the deadline, while achieving collision avoidance. We run the same QP with the reach constraint replaced by the standard CBF condition $L_f h + L_g h\, u \geq -\alpha h$. Collision constraints, input bounds, and PD reference are identical. When $T=20s$, the PD reference is utilized mostly, yielding straighter trajectories. When the deadline is short, $T=17s$, the constricting CBF uses more control intervention to drive the agents into the target set by the deadline. The standard CBF is unable to reach within the deadline of $T=20s$.

\begin{figure}[t]
  \centering
  \begin{tikzpicture}
  \node at (-2,-0.4) {\includegraphics[width=0.39\linewidth]{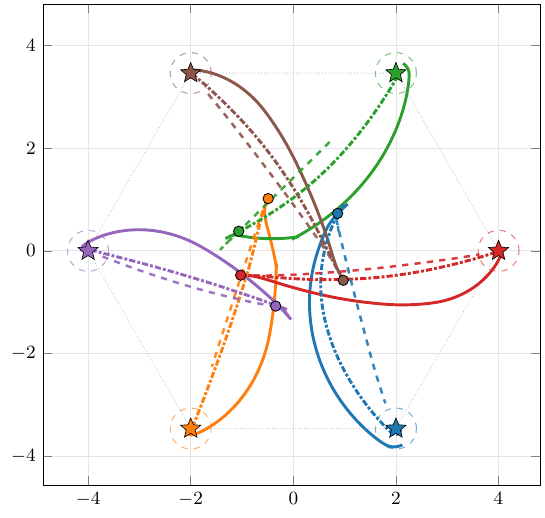}};
  \node at (-2,-2.2) {$p_{x,i}$};
  \node at (-4.0,-0.3) [rotate=90] {$p_{y,i}$};
  \draw [thick, blue, line width=1.0pt, decorate,
           decoration={snake, segment length=10mm, amplitude=1.5mm}]
           {(1.0,0.6) -- (2.8,0.6)};

  \node at (1.9,1.0) {$T=17$};

\draw [thick, purple, densely dashdotted, line width=1.0pt, decorate,
           decoration={snake, segment length=10mm, amplitude=1.5mm}]
           {(1.0,-0.5) -- (2.8,-0.5)};       

  \node at (1.85,-0.1) {$T=20$};

  \draw [thick, red, dashed, line width=1.0pt, decorate,
           decoration={snake, segment length=10mm, amplitude=1.5mm}]
           {(1.0,-1.6) -- (2.8,-1.6)};       

  \node at (1.85,-1.2) {Standard CBF};  
  
  \end{tikzpicture}
\vspace{-0.0ex}
  \caption{Prescribed-time reach with $N = 6$ unicycles,
  $T = 17,20$s. \emph{Solid:} proposed constricting CBF controller,  \emph{dashed:}
  standard CBF baseline with slack on the reach constraint. All share identical pairwise collision CBFs, input bounds, and PD
  reference. Agents randomly and are assigned to the target stars, with target balls of radius $\epsilon = 0.4$~m. The proposed method enters all reach balls at $t = T$. The baseline fails to enter the
  target set.}
  \vspace{-0.0ex}
  \label{fig:hex-overlay}
\end{figure}

\subsection{Prescribed-time safety via HOCBF constricting tube}
\label{sec:exp2}

We apply the constricting tube framework to a system of relative
degree~2, demonstrating the HOCBF extension of
Remark~\ref{rem:relative_degree}. The system is a planar double
integrator $\dot{p} = v$, $\dot{v} = u$, with state $x = (p_x, p_y,
v_x, v_y)^\top$, control $u = (u_x, u_y)^\top$, $\|u\|_\infty \leq
u_{\max} = 2\,\mathrm{m/s^2}$, and initial condition $x(0) = (3, 2,
-0.3, -0.1)^\top$. The task is to drive the position $p = (p_x, p_y)$ into
$\mathcal{C} = \{x : \|p\| \leq 0.5\}$, by $T = 25\,\mathrm{s}$.

The barrier $h(x) = \epsilon^2 - p_x^2 - p_y^2$ satisfies $L_g h = 0$
and has relative degree~2. 
We use a $\mathcal{C}^2$ quadratic schedule
$r(x(0),t) = (r_0+\delta)(1-t/T)^2 - \delta$,
where $r_0 = -h(x(0)) = 12.75$ and $\delta = \epsilon^2/2 = 0.125$.
The schedule satisfies $r(x(0),0) = r_0$ and $r(x(0),T) = -\delta$,
so the tube terminates $\delta$ inside $\mathcal{C}$, guaranteeing
$h(x(T)) \geq \delta > 0$ rather than merely $h(x(T)) \geq 0$.
The HOCBF sequence is
\begin{align}
  \psi_0(x,t) &= h(x) + r(x(0),t), \\
  \psi_1(x,t) &= L_f h + \dot{r}(x(0),t) + \gamma_1\,\psi_0,
\end{align}
with $\gamma_1 = \gamma_2 = 0.9$. The condition
$\dot{\psi}_1 + \gamma_2\psi_1 \geq 0$ yields the QP constraint
\begin{equation}
\begin{aligned}
  L_f^2 h + L_g L_f h\,u + \ddot{r}(x(0),t)
  &+ \gamma_1(L_f h + \dot{r}(x(0),t)) \\
  & \ \qquad \qquad + \gamma_2\,\psi_1  \geq 0.
  \end{aligned}
  \label{eq:HOCBF_exp2}
\end{equation}

We demonstrate the use of our framework as a safety filter. A nominal PD controller
\begin{equation}
  u_{\mathrm{nom}}(x) = -k_p\,p - k_d\,v,
  \qquad k_p = 0.01,\;\; k_d = 0.05,
\end{equation}
provides a stabilizing tendency toward the origin, and the HOCBF QP
minimally deviates from it as follows:
\begin{equation}
  \min_{u}\;\|u - u_{\mathrm{nom}}\|^2
  \quad\text{s.t.}\quad
  \eqref{eq:HOCBF_exp2},\;\;
  \|u\|_\infty \leq u_{\max}.
\end{equation}
The PD gains are deliberately slow: the nominal controller alone does
not drive $x$ into $\mathcal{C}$ by $T = 25\,\mathrm{s}$. The HOCBF
constraint intervenes to enforce the deadline, with the constricting
schedule providing the inward pressure.

As seen in Figure~\ref{fig:hocbf}, the HOCBF controller achieves $\|p(T)\| = 0.36\,\mathrm{m} <
\epsilon$ with $h(x(T)) = 0.121 > \delta$, confirming prescribed-time
entry into the interior of $\mathcal{C}$. The nominal
controller does not enter $\mathcal{C}$ with time $T$ as seen in Figure~\ref{fig:hocbf}(left), confirming that the HOCBF constraint is the mechanism
responsible for the timing guarantee.

\begin{figure}
\centering
\begin{tikzpicture}
    \node at (-1.6,0)
        {\includegraphics[width=0.4\linewidth]{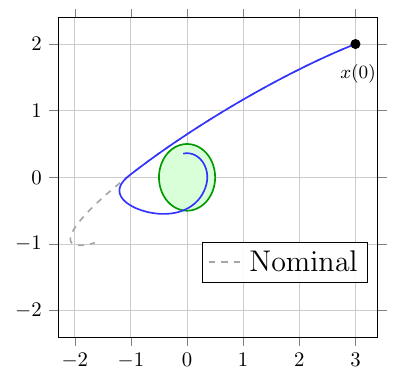}};
    \node at (-1.7,-2.0) {$p_x\,(\mathrm{m})$};
    \node at (-3.4,0) [rotate=90] {$p_y\,(\mathrm{m})$};

    \node at (2.5,0)
        {\includegraphics[width=0.4\linewidth]{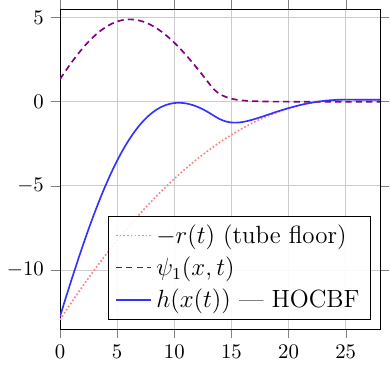}};
    \node at (2.7,-2.0) {Time (s)};
    \node at (0.5,0) [rotate=90] {$\psi_0,\,h(x)$};
\end{tikzpicture}
\caption{Higher order constricting tubes for the double integrator with relative degree~2,
$T = 25\,\mathrm{s}$. \emph{(Left)} Position trajectory: the HOCBF controller
(blue) enters $\mathcal{C}$ by $T$; the nominal PD controller (red)
does not. \emph{(b)} Barrier values $\psi_0(x,t)$ (tube) and $h(x(t))$
(target set): $\psi_0 \geq 0$ is maintained throughout and
$h(x(T)) \geq \delta > 0$ at the deadline.}
\label{fig:hocbf}
\end{figure}


\subsection{Prescribed-time reach-avoid planning for a unicycle}
\label{sec:unicycle}

We demonstrate the constricting tube framework in a trajectory
planning setting where a nonlinear unicycle must reach a target
region by a deadline while avoiding an obstacle that lies directly
on the straight-line path to the target. We embed the constricting
tube and obstacle avoidance as hard path constraints in a model
predictive control problem, solved online with IPOPT via CasADi.
Let $x = (p_x,\, p_y,\, \theta)^\top \in \mathbb{R}^2 \times
\mathbb{S}^1$ denote the position and heading of the unicycle, with
control $u = (v,\, \omega)^\top$,
\begin{equation}
  \dot{p}_x = v\cos\theta,
  \qquad
  \dot{p}_y = v\sin\theta,
  \qquad
  \dot{\theta} = \omega.
  \label{eqn:unicycle}
\end{equation}
Controls are bounded by $|v| \leq v_{\max} = 1.5\,\mathrm{m/s}$ and
$|\omega| \leq \omega_{\max} = 2.0\,\mathrm{rad/s}$. The unicycle
starts at $x(0) = (4.0,\, 3.0,\, \theta_0)^\top$ with $\theta_0 =
\mathrm{atan2}(-3,\,-4)$, pointing directly toward the origin at
distance $\|p(0)\| = 5\,\mathrm{m}$. The \emph{constricting reach
barrier} targets the ball $\mathcal{C}_1 = \{x : h(x) \geq 0\}$
with
\begin{equation}
  h(x) = \epsilon^2 - p_x^2 - p_y^2, \qquad \epsilon = 0.5\,\mathrm{m}.
  \label{eq:h1_unicycle}
\end{equation}
The \emph{static avoidance barrier} keeps the unicycle outside a
circular obstacle of radius $\rho = 0.6\,\mathrm{m}$ centered at
$c = (2.0,\, 1.5)^\top$, which lies on the path from
$x(0)$ to the origin:
\begin{equation}
  h_{\mathrm{obs}}(x) = \|p - c\|^2 - \rho^2.
  \label{eq:h2_unicycle}
\end{equation}
We have $h_1(x(0)) = -24.75 < 0$ and $h_{\mathrm{obs}}(x(0)) = 5.89 > 0$, so the
unicycle starts outside the target but safely clear of the obstacle. We use the  quadratic schedule
\begin{equation}
  r(x(0),t) = r_{0}
            \!\left(1 - \frac{t}{T}\right)^{\!2},
  \label{eq:r1_unicycle}
\end{equation}
with $r_{0} = -h(x(0)) = 24.75$, and deadline
$T = 20\,\mathrm{s}$. 
At each planning instance $t$, we solve the following optimal control problem over the horizon
$[t,\, t + S]$ with $S = 1.5\,\mathrm{s}$:
%
\begin{equation*}
  \begin{aligned}
    \min_{u(\cdot)}\;
      & \int_{t}^{t + S} \!\|u(\tau)\|^2 \,\mathrm{d}\tau
        + \beta \|p(t + S)\|^2 \\
    \mathrm{s.t.}\quad
      & \dot{x}(\tau) = f(x(\tau)) + g\big(x(\tau),u(x(\tau),\tau)\big), \\
      & h\bigl(x(\tau)\bigr) + r(x(0),\tau) \geq 0,  \quad h_{\mathrm{obs}}\bigl(x(\tau)\bigr) \geq 0, \\
      & |v(\tau)| \leq v_{\max},\quad |\omega(\tau)| \leq \omega_{\max}, \quad \forall \tau \in [t,\, t + S].
  \end{aligned}
  \label{eq:NMPC_unicycle}
\end{equation*}
Here $f(x) + g(x, u)$ denotes the unicycle dynamics~\eqref{eqn:unicycle} and $\beta = 10$ penalizes the terminal distance to the origin. The constricting tube constraint evaluates
$r(\tau)$ at time $\tau \in [t, t + S]$,
so the planner is aware of the tightening schedule throughout the
horizon and proactively plans the obstacle detour with the remaining
time budget in mind. The continuous-time problem
is solved using Euler discretization $\mathrm{d}t = 0.005$ s, and
the optimal control over the first interval is applied as the horizon
advances. The planner safely navigates around the obstacle, maintaining
$h_{\mathrm{obs}}(x(\tau)) \geq 0$ at all times, and achieves $h(x(T)) =
+0.233 > \delta = 0.2$, placing the unicycle at $\|p(T)\| =
0.13\,\mathrm{m}$ strictly inside $\mathcal{C}_1$ at the deadline.
Figure~\ref{fig:unicycle} illustrates the trajectory and barrier
histories.

\begin{figure}
\begin{tikzpicture}
    \node at (-1.6,0) {\includegraphics[width = 0.4\linewidth]{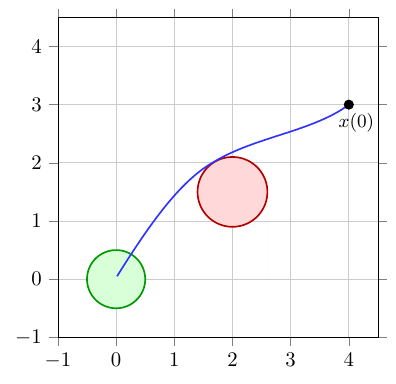}};
    \node at (-1.5,-1.8) {$p_x$};
    \node at (-3.4,0) {$p_y$};
    
    \node at (2.5,0) {\includegraphics[width=0.4\linewidth]{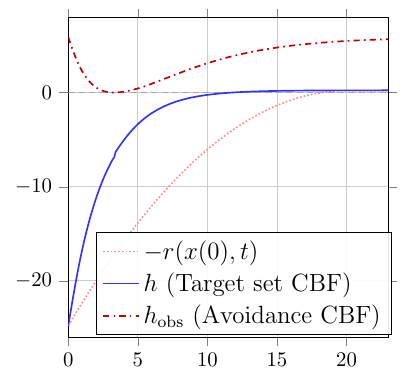}};
    \node at (2.7,-1.8) {Time};
\end{tikzpicture}
\vspace{-1.0ex}
\caption{Prescribed-time reach-avoid planning with the unicycle.
\emph{(Left)} The planner navigates around the obstacle
(red) and enters the target set (green) by the deadline. \emph{(Right)}
Barrier value vs. time: Prescribed-time reach CBF $h(x(t))$ (blue,
solid) rises from $-24.75$ to $+0.233 > \delta$ at $t = T$, staying
above the constricting tube floor $-r(x(0),t)$ (blue, dotted) throughout. Collision avoidance CBF $h_{\mathrm{obs}}(x(t))$ (red, dash-dotted) stays
positive.}
\label{fig:unicycle}
\vspace{-3.0ex}
\end{figure}

\section{Conclusion}

We introduced a constricting CBF framework that reduces prescribed-time
recovery to a forward invariance problem via a designer-specified
constricting tube. The framework yields a single affine constraint
independent of the state dimension, with feasibility conditions
verifiable at design time. Control effort is bounded throughout
the horizon by construction, in contrast to prescribed-time methods
where gains must diverge as the deadline approaches. Validation on a
16-dimensional multi-agent system demonstrates scalability of the
framework, with peak control effort
well below the saturation limit.
A unicycle reach-avoid problem further demonstrates the
modularity of the framework to incorporate multiple simultaneous
constraints. Future directions include robustness to disturbances,
optimal schedule design, and recursive feasibility guarantees under
input constraints.

\bibliography{alias,refs}

@STRING{tac  = "IEEE Transactions on Automatic Control"}

@STRING{lcss  = "IEEE Control Systems Letters"}

@STRING{automatica = "Automatica"}

@STRING{acc = "{A}merican {C}ontrol {C}onference"}

@STRING{ecc = "{E}uropean {C}ontrol {C}onference"}

@STRING{cdc = "{IEEE} Conf.\ on Decision and Control"}

@STRING{icra = "{IEEE} Int.\ Conf.\ on Robotics and Automation"}

@STRING{springer = "Springer"}

@article{KG-etal:2021_multirate,
  title={Multi-rate control design under input constraints via fixed-time barrier functions},
  author={Garg, K. and Cosner, R. K. and Rosolia, U. and Ames, A. D. and Panagou, D.},
  journal=lcss,
  volume={6},
  pages={608--613},
  year={2021}
}

@article{TYH-etal:2024,
  author={Huang, T.-Y. and Zhang, S. and Dai, X. and Capone, A. and Todorovski, V. and Sosnowski, S. and Hirche, S.},
  journal=lcss, 
  title={Learning-Based Prescribed-Time Safety for Control of Unknown Systems With Control Barrier Functions}, 
  year={2024},
  volume={8},
  number={},
  pages={1817-1822}
  }

@article{JSK-SK:2025,
  title={Fixed time convergence guarantees for Higher Order Control Barrier Functions},
  author={S. K., Janani and Kolathaya, S.},
  journal={arXiv preprint arXiv:2507.13888},
  year={2025}
}

@article{ADA-XX-JWG-PD:2016,
  title={Control barrier function based quadratic programs for safety critical systems},
  author={Ames, A. D. and Xu, X. and Grizzle, J. W. and Tabuada, P.},
  journal=tac,
  volume={62},
  number={8},
  pages={3861--3876},
  year={2016}
}

@article{IMM-AMB-CJT:2005,
  title={A time-dependent Hamilton-Jacobi formulation of reachable sets for continuous dynamic games},
  author={Mitchell, I. M. and Bayen, A. M. and Tomlin, C. J.},
  journal=tac,
  volume={50},
  number={7},
  pages={947--957},
  year={2005}
}

@article{AP:2011,
  title={Nonlinear feedback design for fixed-time stabilization of linear control systems},
  author={Polyakov, A.},
  journal=tac,
  volume={57},
  number={8},
  pages={2106--2110},
  year={2011}
}

@article{WX-CB:2021,
  title={High-order control barrier functions},
  author={Xiao, W. and Belta, C.},
  journal=tac,
  volume={67},
  number={7},
  pages={3655--3662},
  year={2021}
}

@ARTICLE{WL-MK:2023,
  author={Li, W. and Krstic, M.},
  journal=tac, 
  title={Prescribed-Time Output-Feedback Control of Stochastic Nonlinear Systems}, 
  year={2023},
  volume={68},
  number={3},
  pages={1431-1446}
}

@article{DG-AA-FP:2026,
  title={Provably Safe Generative Sampling with Constricting Barrier Functions},
  author={Gadginmath, D. and Allibhoy, A. and Pasqualetti, F.},
  journal={arXiv preprint arXiv:2602.21429},
  year={2026}
}

@article{YS-YW-JH-MK:2017,
  title={Time-varying feedback for regulation of normal-form nonlinear systems in prescribed finite time},
    author={Song, Y. and Wang, Y. and Holloway, J. and Krstic, M.},
  journal=automatica,
  volume={83},
  pages={243--251},
  year={2017},
  publisher={Elsevier}
}

@inproceedings{SB-CJT:2021,
  title={Deepreach: A deep learning approach to high-dimensional reachability},
  author={Bansal, S. and Tomlin, C. J.},
  booktitle=icra,
  address={Xi'an, China},
  pages={1817--1824},
  year={2021}
}

@inproceedings{KG-EA-DP:2020_convergenceCLF,
  title={Prescribed-time convergence with input constraints: A control Lyapunov function based approach},
  author={Garg, K. and Arabi, E. and Panagou, D.},
  booktitle=acc,
  address={Denver, Colorado},
  pages={962--967},
  year={2020}
}

@inproceedings{KG-DP:2021,
  title={Robust control barrier and control Lyapunov functions with fixed-time convergence guarantees},
  author={Garg, K. and Panagou, D.},
  booktitle=acc,
  address={New Orleans, Louisiana},
  pages={2292--2297},
  year={2021}
}

@inproceedings{SB-MC-SH-CJT:2017,
  title={Hamilton-jacobi reachability: A brief overview and recent advances},
  author={Bansal, S. and Chen, M. and Herbert, S. and Tomlin, C. J.},
  booktitle=cdc,
  address={Melbourne, Australia},
  pages={2242--2253},
  year={2017}
}

@inproceedings{JJC-etal:2021,
  title={Robust control barrier--value functions for safety-critical control},
  author={Choi, J. J. and Lee, D. and Sreenath, K. and Tomlin, C. J. and Herbert, S. L.},
  booktitle=cdc,
  address={Austin, Texas},
  pages={6814--6821},
  year={2021}
}

@inproceedings{JFF-MC-CJT-SSS:2015,
  title={Reach-avoid problems with time-varying dynamics, targets and constraints},
  author={Fisac, J. F. and Chen, M. and Tomlin, C. J. and Sastry, S. S.},
  booktitle={International Conference on Hybrid Systems: Computation and Control},
  address={Seattle, Washington},
  pages={11--20},
  year={2015}
}

@inproceedings{ADA-SC-etal:2019,
  title={Control barrier functions: Theory and applications},
  author={Ames, A. D. and Coogan, S. and Egerstedt, M. and Notomista, G. and Sreenath, K. and Tabuada, P.},
  booktitle=ecc,
  address={Naples, Italy},
  pages={3420--3431},
  year={2019},
}

@book{FB-SM:08,
  title={Set-theoretic methods in control},
  author={Blanchini, F. and Miani, S.},
  volume={78},
  year={2008},
  publisher={Springer}
}
\bibliographystyle{unsrt}

\end{document}